\documentclass{amsart}
\usepackage{amssymb,latexsym,amsmath,amsfonts}
\usepackage[colorlinks,linkcolor=blue,anchorcolor=blue,citecolor=blue]{hyperref}

\newcommand\C{{\mathbb C}}
\newcommand\F{{\mathbb F}}
\newcommand\Q{{\mathbb Q}}

\newcommand\Z{{\mathbb Z}}

\newcommand\A{{\mathbb A}}

\newcommand\GL{{\mathrm {GL}}}
\newcommand\SL{{\mathrm {SL}}}
\newcommand\Gal{{\mathrm {Gal}}}

\newcommand\HH{{\mathcal H}}
\newcommand\OO{{\mathcal O}}

\newcommand\gal{\mathrm{Gal}}

\newcommand\ns{\mathrm{ns}}

\newcommand\ord{\mathrm{Ord}}
\newcommand\h{\mathrm{h}}

\newtheorem{theorem}{Theorem}[section]
\newtheorem{proposition}[theorem]{Proposition}
\newtheorem{lemma}[theorem]{Lemma}
\newtheorem{corollary}[theorem]{Corollary}

\numberwithin{equation}{section}

\begin{document}
\title{Bounding $j$-invariant of integral points  on $X_{\ns}^{+}(p)$}

\author{Aur\'elien Bajolet}
\address{Institut de Mathematiques de Bordeaux, Universite Bordeaux 1
, 33405 Talence Cedex, France}
%\curraddr{Department of Mathematics and Statistics,
%Case Western Reserve University, Cleveland, Ohio 43403}
\email{Aurelien.Bajolet@math.u-bordeaux1.fr}

\author{Min Sha}
\address{Institut de Mathematiques de Bordeaux, Universite Bordeaux 1
, 33405 Talence Cedex, France}
%\curraddr{Department of Mathematics and Statistics,
%Case Western Reserve University, Cleveland, Ohio 43403}
\email{shamin2010@gmail.com}
\thanks{%The authors were partially supported by the {\it Agence nationale de la recherche} project ``HAMOT".
The second author is supported by China Scholarship Council.}

\subjclass[2010]{Primary 11G16, 11J86; Secondary 14G35, 11G50}

\keywords{Modular curves, Non-split Cartan, $j$-invariant,  Baker's method}

\date{}

\begin{abstract}
For prime $p\ge 7$, by using Baker's method we obtain two explicit bounds in terms of $p$
for the $j$-invariant of an integral point on $X_{\ns}^{+}(p)$ which is
the modular curve of level~$p$ corresponding to the normalizer of a non-split Cartan subgroup of $\GL_2(\Z/p\Z)$.
\end{abstract}

\maketitle

\section{Introduction}

Let~$p$ be a prime number, $p\ge 7$.
We denote by $X_{\ns}^{+}(p)$ the modular curve of level~$p$ corresponding to the normalizer of a non-split Cartan subgroup of $\GL_2(\Z/p\Z)$. See, for instance,  Serre~\cite{Se97}, Section~A.5 for definitions and basic properties. In particular, this curve has a canonical $\Q$-model, which will be used throughout.  One can similarly define $X_{\ns}^{+}(N)$ of a composite level~$N$, but we restrict to prime levels in this article.

We denote by~$j$ the standard $j$-invariant function on $X_{\ns}^{+}(p)$.
We call a rational point ${P\in X_{\ns}^{+}(p)(\Q)}$ an integral point with respect to~$j$ if ${j(P) \in \Z}$.

The modular curve $X_{\ns}^{+}(p)$ has ${(p-1)/2}$ cusps, and all its cusps are conjugate over $\Q$.
Hence, by classical Siegel's finiteness theorem~\cite{Si29}, for ${p\ge 7}$ the curve $X_{\ns}^{+}(p)$ can have only finitely many integral points. Moreover,  as follows from \cite[Proposition~5.1(a)]{Bi95},
their sizes can be bounded effectively in terms of~$p$.

In fact, under ``Runge condition" which roughly says that all the cusps are not conjugate,
there is an explicit bound for the $j$-invariant of the integral points on arbitrary modular curves
over arbitrary number fields,
see \cite[Theorem 1.2]{BP10}. Unfortunately, Runge condition fails for $X_{\ns}^{+}(p)$.
So we must introduce other techniques.

In this paper we use Baker's method, more precisely Baker's inequality in the form due to Matveev \cite[Corollary 2.3]{Ma},
to obtain two explicit bounds in terms of $p$
for the $j$-invariant of an integral point on $X_{\ns}^{+}(p)$.

Our first main result is the following general theorem.

\begin{theorem}
\label{main1}
Assume that ${p\ge 7}$ and let ${d\ge 3}$ be a divisor of  ${(p-1)/2}$.  Then for any integral point~$P$ on $X_{\ns}^{+}(p)$ we have
$$
\log|j(P)| <C(d)p^{6d+5}(\log p)^{2},
$$
where $C(d)=30^{d+5}\cdot d^{-2d+4.5}$.
\end{theorem}

In particular, if we choose $d=\frac{p-1}{2}$ in Theorem \ref{main1}, we obtain a bound which is explicit in~$p$.
\begin{theorem}\label{main2}
\label{tmain}
Assume that ${p\ge 7}$. Then for any integral point~$P$ on $X_{\ns}^{+}(p)$ we have
$$
\log|j(P)| <41993\cdot 13^{p} \cdot p^{2p+7.5}(\log{p})^{2}.
$$
\end{theorem}

By comparing these two theorems, the bound in Theorem \ref{tmain} can be drastically reduced if $\frac{p-1}{2}$ has a small divisor.

The interest in integral points on the modular curves corresponding to normalizers of
Cartan subgroups  is motivated by their relation to imaginary quadratic field of low class number.
See Appendix~A in Serre's book~\cite{Se97} for a nice historical account and further explanations.
In particular, integral points on the curves $X_{\ns}^{+}(24)$ and  $X_{\ns}^{+}(15)$
were studied by Heegner~\cite{He52} and Siegel~\cite{Si68} in their classical work on the
class number~$1$ problem. Kenku~\cite{Ke85} determined all integral points on $X_{\ns}^{+}(7)$,
and Baran~\cite{Ba09,Ba10} did this for $X_{\ns}^{+}(9)$ and $X_{\ns}^{+}(15)$.

A general method for computing integral points on $X_{\ns}^{+}(p)$ is developed in~\cite{BaBi11}.
Much more is known on integral and even rational points on modular curves corresponding
to the normalizers of split Cartan subgroups,  see \cite{BP11,BPR11}.
In particular, the authors in \cite{BP11} used the bound of $j$-invariant of integral points
to solve Serre's uniformity problem in the split Cartan case and finally left this problem with
the non-split Cartan case.

In addition, recently the second author has used a different approach by applying Baker's method to
get some effective bounds for the $j$-invariant of integral points on arbitrary modular curves
 over arbitrary number fields assuming that the number of cusps is not less than 3, see \cite{Sha}.

\section{Notations and conventions}
Through out this paper, $\log$ stands for the principal branch of the complex logarithm, in this case will use the following estimate without special reference
$$
|\log(1+z)|\le\frac{|\log(1-r)|}{r}|z|,
$$
for $|z|\le r<1$, see \cite[Formula (4)]{BP10}.
%The other one is the $p$-adic logarithm function, for example see \cite[Chapter IV Section 2]{Koblitz}.

We fix $p$ a prime number not less than 7. Let $G$ be the normalizer of a non-split Cartan subgroup of $\GL_{2}(\Z/p\Z)$
and $X_{\ns}^{+}(p)$ be the modular curve corresponding to $G$. In fact, up to conjugation, we know
$$
G=\left\{\begin{pmatrix}\alpha &\Xi\beta \\ \beta & \alpha\end{pmatrix},\begin{pmatrix}\alpha &\Xi\beta \\ -\beta & -\alpha\end{pmatrix}
: \alpha,\beta\in \F_{p}, (\alpha,\beta)\ne (0,0)
\right\},
$$
where $\Xi$ is a quadratic non-residue modulo $p$. In particular, one can choose $\Xi=-1$ if $p\equiv 3\mod 4$.
Moreover, $|G|=2(p^{2}-1)$ following from \cite[Formula (2.3)]{Ba10} and $\det G=\F_{p}^{\times}$,
where $\det G$ is the image of $G$ under the determinant map $\det:\GL_{2}(\Z/p\Z)\to \F_{p}^{\times}$.
%So $X_{\ns}^{+}(p)$ is defined over $\Q$.

%Let $X(p)$ be the principal modular curve of level $p$. It is well-known that
%$\gal(\Q(\zeta_{p})(X(p))/\Q(X_{\ns}^{+}(p)))\cong G$. We identify the two groups.

In the sequel, we fix a subgroup $H$ of $\F_{p}^{\times}$ such that $-1\in H$ and $[\F_{p}^{\times}:H]\ge 3$. Put $d=[\F_{p}^{\times}:H]$, then
we have
$$
d \left|\dfrac{p-1}{2} \right. \qquad {\rm and} \qquad d=[K:\Q],
$$
where $K=\Q(\zeta_{p})^{H}$ and $\zeta_{p}=e^{\frac{2\pi i}{p}}$.
We can identify the Galois group $\gal(K/\Q)$ with $\F_{p}^{\times}/H$, we also identify $\gal(\Q(\zeta_{p})/K)$ with $H$.
In particular, $K\subseteq \Q(\zeta_{p})^{+}$, where $\Q(\zeta_{p})^{+}=\Q(\zeta_{p}+\bar{\zeta_{p}})$.

Put
$$
G_{H}=\{g\in G:\det g\in H\}.
$$
Then the determinant map induces an isomorphism: $G/G_H\cong \F_{p}^{\times}/H$. We denote by $X_{H}$ the modular curve corresponding to $G_{H}$, which is defined over $K$.
Here $X_{H}$ and $X_{\ns}^{+}(p)$ have the same geometrically integral model, and the function field of
$X_{H}$ is $K(X_{\ns}^{+}(p))$. The curve $X_{H}$ also has the same cusps as $X_{\ns}^{+}(p)$.
In particular, $\gal(K(X_{H})/\Q(X_{\ns}^{+}(p)))\cong \gal(K/\Q)$.

Hence, in this paper
we identify the following four groups: $\gal(K(X_{H})/\Q(X_{\ns}^{+}(p)))$, $\gal(K/\Q)$, $\F_{p}^{\times}/H$ and $G/G_H$.
The readers should interpret the exact meaning based on the context.

Let $\mathcal{H}$ denote the Poincar$\acute{\rm e}$ upper half-plane: $\mathcal{H}=\{\tau\in\C: {\rm Im} \tau>0\}$.
We put $\bar{\HH}=\HH\cup\Q\cup\{i\infty\}$. We denote by $D$ the standard fundamental domain of $\SL_{2}(\Z)$.
If $\Gamma$ is the pullback of $G_{H}\cap\SL_{2}(\Z/p\Z)$ to $\SL_{2}(\Z)$,
then the set $X_{H}(\C)$ of complex points is analytically isomorphic to the quotient $\bar{\HH}/\Gamma$,
supplied with the properly defined topology and analytic structure.
See any standard reference like \cite{Shimura} for all the missing details.

For  $\textbf{a}=(a_{1},a_{2})\in \Q^{2}$,  we put $\ell_{\textbf{a}}
=B_{2}(a_{1}-\lfloor a_{1} \rfloor)/2$, where  $B_{2}(T)=T^{2}-T+\frac{1}{6}$ is the second Bernoulli polynomial.
Obviously $|\ell_{\textbf{a}}|\le 1/12$, this will be used without special reference.

We put ${\A=\left(p^{-1}\Z/\Z \right)^{2}\setminus \{(0,0)\}}$. In this paper, we also identify $p^{-1}\Z/\Z$ with $p^{-1}\F_{p}$.
Moreover we always choose a representative element of
 ${\textbf{a}=\left(a_{1},a_{2}\right)}\in (p^{-1}\Z/\Z)^{2}$ satisfying ${0 \le a_{1},a_{2} <1}$.
 So in the sequel for every $\textbf{a}\in (p^{-1}\Z/\Z)^{2}$, we have $\ell_{\textbf{a}}=B_{2}(a_{1})/2$.

In the sequel, we use the notation $O_{1}(\cdot)$. Precisely, $A=O_{1}(B)$ means that $|A|\le B$.

\section{Preparations}

\subsection{Siegel functions and modular units}
Let $\textbf{a}=(a_{1},a_{2})\in \Q^{2}$ be such that $\textbf{a}\not\in\Z^{2}$, and let $g_{\textbf{a}}$ be the corresponding
\textit{Siegel function}, see \cite[Section 2.1]{KL81}. We have the following infinite product presentation for $g_{\textbf{a}}$,
see \cite[Formula (7)]{BP10}.
$$
g_{\textbf{a}}(\tau)=-q_{\tau}^{B_{2}(a_{1})/2}e^{\pi ia_{2}(a_{1}-1)}\prod\limits_{n=0}^{\infty}\left(1-q_{\tau}^{n+a_{1}}e^{2\pi ia_{2}}\right) \left(1-q_{\tau}^{n+1-a_{1}}e^{-2\pi ia_{2}}\right),
$$
where $q_{\tau}=e^{2\pi i \tau}$ and $B_{2}(T)=T^{2}-T+\frac{1}{6}$ is the second Bernoulli polynomial.

From the proof of \cite[Proposition 2.3]{BP10} and replacing
$3|q_{\tau}|$ by $2.03|q_{\tau}|$ in \cite[Formula (11)]{BP10}, we get directly the following lemma.
\begin{lemma}\label{ga}
Let ${\rm\bf a}\in \Q^{2}\setminus \Z^{2}$. Then for $\tau\in D$, we have
\begin{equation}
\log|g_{{\rm\bf a}}(\tau)|=\ell_{{\rm\bf a}}\log|q_{\tau}|+\log|1-q_{\tau}^{a_{1}}e^{2\pi ia_{2}}|+\log|1-q_{\tau}^{1-a_{1}}e^{-2\pi ia_{2}}|+O_{1}(2.03|q_{\tau}|).
\notag
\end{equation}
\end{lemma}

Recall that by a modular unit on a modular curve we mean a rational function having poles and zeros only at the cusps.

For $\textbf{a}\in (p^{-1}\Z)^{2}\setminus \Z^{2}$, we denote $g_{\textbf{a}}^{12p}$ by $u_{\textbf{a}}$, which
is a modular unit on the principal modular curve $X(p)$ of level $p$.
Moreover, we have $u_{\textbf{a}}=u_{\textbf{a}^{\prime}}$ when $\textbf{a}\equiv \textbf{a}^{\prime}$ mod $\Z^{2}$.
Hence, $u_{\textbf{a}}$ is well-defined when ${\textbf{a}}\in\A$. In addition, every $u_{\textbf{a}}$ is integral over $\Z[j]$.
For more details, see \cite[Section 4.2]{BP10}.

Furthermore, the Galois action on the set $\{u_{\textbf{a}}\}$ is compatible with the right linear action of
$\GL_{2}(\Z/p\Z)$ on it. That is, for any $\sigma\in\Gal(\Q(X(p))/\Q(X(1)))\cong\GL_{2}(\Z/p\Z)/\pm 1$ and any $\textbf{a}\in\A$, we have
$$
u_{\rm\bf a}^{\sigma}=u_{\rm\bf a\sigma}.
$$

Here we borrow a result and its proof from \cite{BaBi11} for the conveniences of readers.
\begin{lemma}[\cite{BaBi11}]\label{u1}
We have
$$
\prod\limits_{{\rm\bf a}\in\A}u_{{\rm\bf a}}=p^{12p}.
$$
\end{lemma}
\begin{proof}
We denote by $u$ the left-hand side of the equality.
Since the set $\A$ is stable with respect to $\GL_{2}(\Z/p\Z)$, $u$
is stable with respect to the Galois action over the field $\Q(X(1))=\Q(j)$. So $u\in \Q(j)$.
Moreover, since $u$ is integral over $\Z[j]$, $u\in \Z[j]$. Notice that
$X(1)$ has only one cusp and
$u$ has no zeros and poles outside the cusps,
so $u$ must be a constant and $u\in\Z$. Since
$$
\sum\limits_{(a_{1},a_{2})\in\A}B_{2}(a_{1})=0  \qquad \text{and } \qquad \sum\limits_{(a_{1},a_{2})\in\A}a_{2}(1-a_{1})=\frac{p^{2}-1}{4},
$$
we have for $q=0$,
\begin{align*}
u&=\prod\limits_{(a_{1},a_{2})\in\A, a_{1}=0}(1-e^{2\pi ia_{2}})^{12p} = \prod\limits_{1\le k<p}(1-e^{2k\pi i/p})^{12p} =p^{12p}.
\end{align*}
\end{proof}

\subsection{$X_{\ns}^{+}(p)$ and $X_{H}$ }\label{X_H}

It is well-known that the curve $X_{\ns}^{+}(p)$ has $(p-1)/2$ cusps. Moreover, these cusps
correspond to the orbits of the (left) action of $G\cap \SL_{2}(\Z/p\Z)$ on
the set $\F_{p}^{2}\setminus \{{0 \choose 0}\}$, see \cite[Lemma 2.3]{BI11}. By definition, these orbits are the sets $\mathcal{L}_{a}$,
defined by $x^{2}-\Xi y^{2}=\pm a$, where $a$ runs through $\F_{p}^{\times}/\{\pm 1\}$, the cusp at infinity corresponds to $a=1$.

Form now on, we fix an integral point $P$ of $X_{\ns}^{+}(p)$ and assume that $|j(P)|>3500$. %Then $P$ is also an integral point of $X_{H}$.
Since every integral point of $X_{\ns}^{+}(p)$ is also an integral point of $X_{H}$,
 $P$ is also an integral point of $X_{H}$.
Hence for our purposes, we only need to focus on the modular curve $X_{H}$.

Notice that since all the cusps have ramification index $p$ in the natural covering ${X_{\ns}^{+}(p)\to X(1)}$, so as the natural covering ${X_{H}\to X(1)}$.

We fix a uniformization $X_{H}(\C)=\bar{\HH}/\Gamma$, and let $\tau_{0}\in \bar{\HH}$ be
a lift of $P$. Pick $\sigma_{c}\in \SL_{2}(\Z)$ such that $\tau=\sigma_{c}^{-1}(\tau_{0})\in D$.
As the proof of \cite[Proposition 3.1]{BP10}, we can choose the cusp $c=\sigma_{c}(i\infty)$
and construct $\Omega_{c}=\sigma_{c}(\Delta)/\Gamma$. Furthermore, for the cusp $c$, following \cite[Section 3]{BP10} let $t_{c}$ be its local parameter and put $q_{c}=t_{c}^{p}$, then $q_{c}$ and $t_c$ are defined and analytic on $\Omega_{c}$. Moreover, $q_c(P)=q_{\tau}$.

%Then $c$ is called a nearby cusp to $P$ with respect to the classical absolute value.

According to \cite[Proposition 3.1]{BP10}, we have
\begin{equation}\label{j(P)0}
\frac{1}{2}|j(P)|\le |q_{c}(P)^{-1}|\le \frac{3}{2}|j(P)|.
\end{equation}
We will use (\ref{j(P)0}) several times without special reference.

In the sequel we can assume that $|q_{c}(P)|\le 10^{-p}$. Indeed, the inequality ${|q_{c}(P)|> 10^{-p}}$ yields a much better estimate
for $\log|j(P)|$ than those given in Section 1.

%From now on we write $q_{c}$ for $q_{c}(P)$. Indeed, $q_{c}(P)=e^{2\pi i\tau}$.

\subsection{Modular units on $X_{H}$}\label{X_H1}
The group $\GL_{2}(\F_p)$ acts naturally (on the right) on the set $\A$.
Since $G_{H}\subset\GL_{2}(\F_p)$, let us consider the natural right group action of $G_{H}$ on $\A$.
 There are $d$ orbits of this group action. These orbits are the sets $\OO_{a}$,
defined by $\{(x/p,y/p): x^{2}-\Xi^{-1} y^{2}\in aH\}$,
where $a$ runs through $\F_{p}^{\times}/H$. In fact, if $(x,y)\in \OO_{a}$, then for any $g\in G_{H}$,
noticing the two possible representations of $g$, it is straightforward to show that $(x,y)\cdot g\in \OO_{a}$.

Based on our conventions in Section 2, we consider the natural right group action of $\gal(K/\Q)$ on the set of orbits of the group action $\A/G_{H}$. Moreover, for any $\sigma\in\gal(K/\Q)$ and any orbit $\OO_{a}$, we have
$$
\OO_{a}\sigma=\OO_{a\sigma}.%\{(x/p,y/p):(x,y)\in\F_{p}^{2}\setminus (0,0), x^{2}-\Xi^{-1} y^{2}\in abH\}.
$$
It is easy to see that this group action is transitive.
So we obtain the following lemma.
\begin{lemma}\label{orbit}
We have $|\OO_{a}|=(p^{2}-1)/d$.
\end{lemma}

Let $S$ be any subset of $\A$, we define
$$
u_{S}=\prod\limits_{\textbf{a}\in S}u_{\textbf{a}}.
$$
Let $\OO$ be an orbit of $\A/G_{H}$, we have
\begin{equation}\label{uo}
u_{\OO}=\prod\limits_{\textbf{a}\in\OO}u_{\textbf{a}}.
\end{equation}
By \cite[Proposition 4.2 (ii)]{BP10}, $u_{\OO}$ is a rational function on the modular curve $X_{H}$.
Furthermore, $u_{\OO}$ is a modular unit on $X_{H}$.

We denote by $\ord_{c}(u_{\OO})$ the vanishing order of $u_{\OO}$ at $c$.
The following lemma is derived directly from Lemma \ref{ga} and \cite[Proposition 4.2 (iii)]{BP10}.
\begin{lemma}\label{uo1}
We have
\begin{equation}\label{uo2}
\log|u_{\OO}(P)|=\frac{\ord_{c}(u_{\OO})}{p}\log|q_{c}(P)|+\log|\gamma_{c}| +O_{1}(17p^{3}|q_{c}(P)|^{1/p})\\
\end{equation}
where
$$\ord_{c}(u_{\OO})=12p^{2}\sum\limits_{{\rm\bf a}\in \OO\sigma_{c}}\ell_{{\rm\bf a}}
\quad {\it and} \quad {\gamma_{ c}=\prod\limits_{\substack{(a_{1},a_{2})\in \OO\sigma_{c}\\ a_{1}=0}}(1-e^{2\pi ia_{2}})^{12p}}.$$
\end{lemma}
\begin{proof}
Here we use the following identity:
$$
u_{\OO}(P)=u_{\OO}(\tau_{0})=u_{\OO}(\sigma_c(\sigma_c^{-1}(\tau_0)))=u_{\OO\sigma_c}(\tau).
$$

Notice that  for $|z|\le r<1$, we have
$$
|\log|1+z||\le\frac{-\log(1-r)}{r}|z|.
$$
Taking $r=0.1$ and combining Lemma \ref{ga} with Lemma \ref{orbit}, we have
\begin{align*}
\log|u_{\OO}(P)|=&\frac{\ord_{c}(u_{\OO})}{p}\log|q_{c}(P)|+\log\left|\gamma_{c}\right|\\
+&O_{1}\left(26 p \dfrac{p^{2}-1}{d} |q_{c}(P)|^{1/p} +25p \dfrac{p^{2}-1}{d}|q_{c}(P)|\right).
\end{align*}
Then this lemma follows from $d\ge 3$.
\end{proof}

We want to indicate that $\gamma_{c}$ is a real algebraic number. Because if $(0,a_{2})\in \OO\sigma_{c}$,
then we have $(0,-a_{2})\in \OO\sigma_{c}$ based on the fact that if $(x,y)\in\OO$,
then $(-x,-y)\in\OO$.

\begin{lemma}\label{rank}
The group generated by the principal divisor $(u_{\OO})$, where $\OO$ runs over the orbits of $\A/G_{H}$,
is of rank $d-1$.
\end{lemma}
\begin{proof}
 By Lemma \ref{u1}, the rank of the free abelian group $(u_{\OO})$ is at most $d-1$. Then
Manin-Drinfeld theorem, as stated in \cite{KL81}, tells us that this rank is maximal possible.
\end{proof}

\section{Baker's method on $X_{H}$}
\label{bakermeth}
In this section we obtain a bound for $\log|j(P)|$, involving various parameters.
Recall that $P$ is the integral point of $X_{\ns}^{+}(p)$ fixed in Section \ref{X_H}.

\subsection{Cyclotomic units}
We introduce a set of independent cyclotomic units of $\Q(\zeta_{p})^{+}$ as follows,
$$
\xi_{k-1}=\zeta_{p}^{(1-k)/2}\cdot  \dfrac{1-\zeta_{p}^{k}}{1-\zeta_{p}}
=\frac{\bar{\zeta_{p}}^{k/2}-\zeta_{p}^{k/2}}{\bar{\zeta_{p}}^{1/2}-\zeta_{p}^{1/2}}, \qquad k=2,\dots, \dfrac{p-1}{2},
$$
for details see \cite[Lemma 8.1]{Wa}.
In particular, $\{-1, \xi_{1},\cdots,\xi_{\frac{p-3}{2}}\}$ is a set of
independent generators for the full group of cyclotomic units of $\Q(\zeta_{p})^{+}$.
Let $m^{\prime}$ be the index of $\langle\xi_{1}, \cdots, \xi_{\frac{p-3}{2}}\rangle$
in the full unit group of $\Q(\zeta_{p})^{+}$ modulo roots of unity,
which is equal to the class number of $\Q(\zeta_{p})^{+}$.

We put
$$
\eta_{k}=\mathcal{N}_{\Q(\zeta_{p})^{+}/K}(\xi_{k})
=\prod\limits_{\sigma\in\gal(\Q(\zeta_{p})^{+}/K)}\xi_{k}^{\sigma}, \qquad k=1,\dots, \dfrac{p-3}{2}.
$$
Let $m$ be the exponent of $\langle\eta_{1}, \cdots, \eta_{\frac{p-3}{2}}\rangle $ in the full unit group of $K$ modulo roots of unity.
Since $[\Q(\zeta_{p})^{+}:K]=|H|/2=\frac{p-1}{2d}$, we have
\begin{equation}\label{index}
m\left|\frac{m^{\prime}(p-1)}{2d}. \right.
\end{equation}
Since $m$ is finite and the rank of the full unit group of $K$ is $d-1$,
the group $\langle\eta_{1}, \cdots, \eta_{\frac{p-3}{2}}\rangle$ modulo roots of unity has rank $d-1$.
In particular, in the sequel we assume that $\eta_{1},\cdots,\eta_{d-1}$ are multiplicatively independent
without loss of generality.

\subsection{More about modular units on $X_{H}$}

We fix  an orbit $\OO$ of the group action $\A/G_{H}$.
Put $U=u_{\OO}$, where $u_{\OO}$ is defined in (\ref{uo}).

Based on our conventions in Section 2, for any $\sigma\in\Gal(K/\Q)$, we can define
$U^{\sigma}$ as the natural Galois action. Indeed, we can view $\sigma$ as an element of $\gal(K(X_{H})/\Q(X_{\ns}^{+}(p)))$
 and $U\in K(X_{H})$.
Moreover, we have $U^{\sigma}=u_{\OO\sigma}$ and $U(P)^{\sigma}=U^{\sigma}(P)$.

Since the Galois group $\Gal(K/\Q)$ acts transitively on the set of orbits of $\A/G_{H}$, we can
rewrite Lemma \ref{u1} as follows.
\begin{lemma}\label{u2}
We have
$$
\prod\limits_{\sigma\in \Gal(K/\Q)}U^{\sigma}=p^{12p}.
$$
\end{lemma}

By Lemma \ref{orbit} and the formula for $\ord_{c}u_{\OO}$ appearing in Lemma \ref{uo1} we obtain a bound for the vanishing order of $U$ at $c$.
\begin{lemma}\label{order}
We have
$$
\left|\ord_{c} U\right|\le \frac{p^{2}(p^{2}-1)}{d}.
$$
\end{lemma}

For $1-\zeta_{p}$, we take the $\Q(\zeta_{p})/K$-norm, setting $\mu=\mathcal{N}_{\Q(\zeta_{p})/K}(1-\zeta_{p})$.
%\mu=\mathcal{N}_{\Q(\zeta_{p})/K}(1-\zeta_{p})=\prod\limits_{\sigma\in\gal(\Q(\zeta_{p})/K)}(1-\zeta_{p})^{\sigma}.
\begin{lemma}
We have $\left( U(P)\right)=\left( \mu^{12p}\right)$.
\end{lemma}
\begin{proof}
Since $P$ is an integral point of $X_{H}$, by \cite[Proposition 4.2 (i)]{BP10} and
Lemma \ref{u2}, the principal ideal $\left( U(P) \right) $
is an integral ideal of the field $K$ of the form $\mathfrak{p}^{n}$,
where $\mathfrak{p}=\left( \mu\right)$ and $n$ is a positive integer.

In addition, since $\mathfrak{p}$ is stable under the Galois action over $\Q$,
we have  $\left( U^{\sigma}(P)\right)=\mathfrak{p}^{n} $ for every $\sigma\in\Gal(K/\Q)$.
Noticing that $\mathfrak{p}^{d}=\left( p\right)$, it follows from Lemma \ref{u2} that $n=12p$.
\end{proof}

So Dirichlet's unit theorem gives
$$
U(P)^{m}=\pm\eta_{0}^{m}\eta_{1}^{b_{1}}\dots \eta_{d-1}^{b_{d-1}},
$$
where $\eta_{0}=\mu^{12p}$ and $b_{1},\cdots,b_{d-1}$ are some rational integers.

Let
$$
V=U/\eta_{0},
$$
then we have
$$
V(P)^{m}=\pm\eta_{1}^{b_{1}}\dots \eta_{d-1}^{b_{d-1}},
$$
and $\ord_{c}V=\ord_{c}U$. For every $\sigma\in \Gal(K/\Q)$, we have
\begin{equation}\label{V(P)}
V^{\sigma}(P)^{m}=\pm(\eta_{1}^{\sigma})^{b_{1}}\dots (\eta_{d-1}^{\sigma})^{b_{d-1}},
\end{equation}
where $V^{\sigma}=U^{\sigma}/\eta_{0}^{\sigma}$. Furthermore, by (\ref{uo2}), we have
\begin{equation}\label{V1}
\log|V^{\sigma}(P)|=\frac{\ord_{c}V^{\sigma}}{p}\log|q_{c}(P)|+\log|\Upsilon_{c,\sigma}|+O_{1}\left(17p^{3}|q_{c}(P)|^{1/p}\right),
\end{equation}
where $\Upsilon_{c,\sigma}=\gamma_{c,\sigma}/\eta_{0}^\sigma$ and
$$
\gamma_{c,\sigma}=\prod\limits_{\substack{(a_{1},a_{2})\in \OO{\sigma\sigma_{c}}\\a_{1}=0}}(1-e^{2\pi ia_{2}})^{12p}.
$$
Notice that $\gamma_{c,\sigma}=\gamma_{c}$ when $\sigma$ is the identity. So $\Upsilon_{c,1}=\gamma_{c}/\eta_{0}$.

Finally we put
$$
B=\max\{|b_{1}|,\cdots,|b_{d-1}|,m\}.
$$

\subsection{Upper bound for B}

We fix an order on the elements of the Galois group by supposing
$$
\gal(K/\Q)=\{\sigma_{0}=1, \sigma_{1}, \cdots, \sigma_{d-1}\}.
$$
Since the real algebraic numbers $\eta_{1},\cdots, \eta_{d-1}$ are multiplicatively independent, the $(d-1)\times (d-1)$ real matrix
$A=\left(\log|\eta_{\ell}^{\sigma_{k}}|\right)_{1\le k,\ell\le d-1}$ is non-singular. Let $\left(\alpha_{k\ell}\right)_{1\le k,\ell\le d-1}$
be the inverse matrix. Then by (\ref{V(P)}) we have
$$
b_{k}=m\sum\limits_{\ell=1}^{d-1}\alpha_{k\ell}\log|V^{\sigma_{\ell}}(P)|, \quad 1\le k\le d-1.
$$

Define the following quantities:
\begin{align*}
&\delta_{c,k}=\frac{m}{p}\sum\limits_{\ell=1}^{d-1}\alpha_{k\ell}\ord_{c}V^{\sigma_{\ell}}, \\
&\beta_{c,k}=m\sum\limits_{\ell=1}^{d-1}\alpha_{k\ell}\log|\Upsilon_{c,\sigma_\ell}|,\\
&\kappa=\max\{\max_{k}\sum\limits_{\ell=1}^{d-1}|\alpha_{k\ell}|,1\}.
\end{align*}
According to (\ref{V1}), we have
\begin{equation}
b_{k}=\delta_{c,k}\log|q_{c}(P)|+\beta_{c,k}+O_{1}\left(17p^{3}m\kappa|q_{c}(P)|^{1/p}\right).
\notag
\end{equation}

Let $\delta=\max\limits_{k}|\delta_{c,k}|$ and $\beta=\max\limits_{k}|\beta_{c,k}|$.
 Then we have
\begin{equation}\label{B}
B\le \delta\log|q_{c}(P)^{-1}|+\beta+2p^{3}m\kappa.
\end{equation}

\subsection{Preparation for Baker's inequality}

We define the following function
\begin{equation}
W=\left\{ \begin{array}{ll}
                 V & \textrm{if $\ord_{c}V=0$},\\
                  \\
                 V^{\ord_{c}V^{\sigma}}(V^{\sigma})^{-\ord_{c}V} & \textrm{if $\ord_{c}V\ne 0$},
                 \end{array} \right.
\notag
\end{equation}
where $\sigma\in \gal(K/\Q)$ and $\sigma\ne 1$. So we always have $\ord_{c}W=0$. Moreover, $W$ is not a constant by Lemma \ref{rank}. 
In Section \ref{specialcase} we will choose special $U$ (i.e. $V$) and $\sigma$ to deal with the exceptional case. 

Define
\begin{equation}
\alpha_{d}=\left\{ \begin{array}{ll}
                 |\Upsilon_{c,1}|^{-1} & \textrm{if $\ord_{c}V=0$},\\
                  \\
                 \left|\frac{\Upsilon_{c,1}^{\ord_{c}V^{\sigma}}}{\Upsilon_{c,\sigma}^{\ord_{c}V}}\right|^{-1} & \textrm{if $\ord_{c}V\ne 0$}.
                 \end{array} \right.
\notag
\end{equation}
Then by (\ref{V1}) and Lemma \ref{order} we obtain
\begin{equation} \label{W(P)}
\log|W(P)|=-\log\alpha_{d}+O_{1}\left(12p^{7}|q_{c}(P)|^{1/p}\right).
\end{equation}

Put
 $$
 \Lambda=m\log|W(P)|+m\log\alpha_{d}.
 $$

If $\ord_{c}V=0$, by (\ref{V(P)}), we have
$$
\Lambda=b_{1}\log|\eta_{1}|+\cdots+b_{d-1}\log|\eta_{d-1}|+m\log\alpha_{d}.
$$
In this case, we put $\alpha_{k}=|\eta_{k}|$ for $1\le k\le d-1$.

If $\ord_{c}V\ne 0$, by (\ref{V(P)}), we have
$$
\Lambda=b_{1}\log\left|\frac{\eta_{1}^{\ord_{c}V^{\sigma}}}{(\eta_{1}^{\sigma})^{\ord_{c}V}}\right|+\cdots+
b_{d-1}\log\left|\frac{\eta_{d-1}^{\ord_{c}V^{\sigma}}}{(\eta_{d-1}^{\sigma})^{\ord_{c}V}}\right|+m\log\alpha_{d}.
$$
In this case, we put $\alpha_{k}=\left|\frac{\eta_{k}^{\ord_{c}V^{\sigma}}}{(\eta_{k}^{\sigma})^{\ord_{c}V}}\right|$ for $1\le k\le d-1$.

Hence, in both two cases we have
\begin{equation}\label{Lambda}
\Lambda=b_{1}\log\alpha_{1}+\cdots+b_{d-1}\log\alpha_{d-1}+m\log\alpha_{d}.
\end{equation}
Notice that all $\alpha_{k}$, $1\le k\le d$, are contained in $\Q(\zeta_{p})^{+}$.

\subsection{Using Baker's inequality}\label{Using Baker}

If $\Lambda=0$, we can get a better bound for $\log|j(P)|$, see the Section \ref{specialcase}.
So here we assume that $\Lambda\ne 0$.

Using \cite[Corollary 2.3]{Ma} and combining (\ref{B}) and (\ref{W(P)}), we have
\begin{equation}
\label{Baker}
\exp(-C_{1}(d)\Omega(\frac{p-1}{2})^{2}(1+\log \frac{p-1}{2})(1+\log B))<|\Lambda|\le \lambda|q_{c}(P)|^{1/p}\le
\lambda\exp\left(\frac{-B+\beta+2p^{3}m\kappa}{\delta p}\right),
\end{equation}
where
\begin{align*}
&C_{1}(d)=\min\left\{\frac{e}{2}d^{4.5}30^{d+3}, 2^{6d+20}\right\},\\
&A_{k}\ge \max\{\frac{p-1}{2}\h(\alpha_{k}),|\log \alpha_{k}|, 0.16\}, 1\le k\le d,\\
&\Omega=A_{1}\cdots A_{d},\quad \lambda=12p^{7}m,
\end{align*}
and $\h(\cdot)$ is the usual absolute logarithmic height.

We obtain $B\le K_{1}\log B+K_{2}$, where
\begin{align*}
&K_{1}=\delta pC_{1}(d)\Omega(\frac{p-1}{2})^{2}(1+\log \frac{p-1}{2}),\\
&K_{2}=\delta pC_{1}(d)\Omega(\frac{p-1}{2})^{2}(1+\log \frac{p-1}{2})+\beta+ 2p^{3}m\kappa+\delta p\log \lambda.
\end{align*}
By \cite[Lemma 2.3.3]{BH1}, we obtain
$$
B\le B_{0}=2(K_{1}\log K_{1}+K_{2}).
$$

Then by (\ref{Baker}), we have
$$
|q_{c}(P)^{-1}|<\lambda^{p}\exp(pC_{1}(d)\Omega (\frac{p-1}{2})^{2}(1+\log \frac{p-1}{2})(1+\log B_{0})).
$$

Finally we have
\begin{equation}\label{j(P)}
\log|j(P)|<pC_{1}(d)\Omega(\frac{p-1}{2})^{2}(1+\log \frac{p-1}{2})(1+\log B_{0})+p\log\lambda+\log 2.
\end{equation}

Hence, to get a bound for $\log|j(P)|$, we only need to calculate the quantities in the above inequality, and
we will do this in the next section.

It is easy to see that
$$
C_{1}(d)=\min\left\{\frac{e}{2}d^{4.5}30^{d+3}, 2^{6d+20}\right\}
<2d^{4.5}30^{d+3}.
$$

\section{Computations}

\subsection{Upper Bound for $m$ }

Let $h^{+}$, $R^{+}$ and $D^{+}$ be the class number, regulator and discriminant of $\Q(\zeta_{p})^{+}$, respectively.

By \cite[Lemma 8.1 and Theorem 8.2]{Wa}, we have $m^{\prime}=h^{+}$. By \cite[Proposition 2.1 and Lemma 4.19]{Wa}, we have $|D^{+}|=p^{\frac{p-3}{2}}$. Then the class number formula (see \cite[Page 37]{Wa}) gives
$$
h^{+}=\left(\frac{p}{4}\right)^{\frac{p-3}{4}}\cdot\frac{1}{R^{+}}\prod\limits_{\chi\ne 1}L(1,\chi).
$$
Using \cite[Theorem 2]{CF} to the field extension $\Q(\zeta_{p})^{+}/\mathbb{Q}$, we have $R^{+}>0.32$.
Applying \cite[Theorem 1]{Lou} to the field extension $\Q(\zeta_{p})^{+}/\mathbb{Q}$ and
noticing the constant $\mu_{\mathbb{Q}}$ below Formula (6) of \cite{Lou}, we get
$$
|L(1,\chi)|<\frac{1}{2}\log p+0.03<\log p,\quad {\rm if\,} \chi\ne 1.
$$
Hence we have
$$
h^{+}< p^{\frac{p-3}{4}}(\log p)^{\frac{p-3}{2}}.
$$

Finally by (\ref{index}), we obtain
\begin{equation}\label{m}
m\le \frac{h^{+}(p-1)}{2d}<p^{\frac{p+1}{4}}(\log p)^{\frac{p-3}{2}}.
\end{equation}

In the sequel we use the following formulas. For any $n\in \mathbb{Z}$ and $a_{1},\cdots,a_{k},\alpha\in \bar{\mathbb{Q}}$, we have
\begin{align*}
&\h(a_{1}+\cdots+a_{k})\le \h(a_{1})+\cdots+\h(a_{k})+\log k,\\
&\h(a_{1}\cdots a_{k})\le \h(a_{1})+\cdots+\h(a_{k}),\\
&\h(\alpha^{n})=|n|\h(\alpha),\\
& \h(\zeta)=0 \textrm{\quad for any root of unity $\zeta\in \C$}.
\end{align*}

\subsection{Height of $\eta_{k-1}$ for $k=2,\dots, (p-1)/2$}
\label{height}

Let $a\in \F_{p}^{\times}$ and $\sigma_{a}\in \gal\left(\Q(\zeta_{p})/\Q\right)$ induced by the automorphism of
$\Q(\zeta_{p}):\zeta_{p}\to \zeta_{p}^{a}$.
%Then ${\gal(\Q(\zeta_{p})^{+}/\Q)=\{\sigma_{1},\cdots,\sigma_{(p-1)/2}\}}$.
%For $1\le a\le \frac{p-1}{2}$,

Since $\xi_{k-1}^{\sigma_{a}}=\frac{\bar{\zeta_{p}}^{ak/2}-\zeta_{p}^{ak/2}}{\bar{\zeta_{p}}^{a/2}-\zeta_{p}^{a/2}}$, we have
$\h(\xi_{k-1}^{\sigma_{a}})\le 2\log 2$.
So
\begin{equation}
\h(\eta_{k-1}^{\sigma_{a}})\le \frac{(p-1)\log 2}{d}.\notag
\end{equation}

Notice that if $-\frac{\pi}{2}<x<\frac{\pi}{2}$, then $\frac{\sin x}{x}>\frac{2}{\pi}$.
Since $\xi_{k-1}^{\sigma_{a}}=\frac{\sin(\pi ak/p)}{\sin(\pi a/p)}$,  we have
$$
|\xi_{k-1}^{\sigma_{a}}|\le \frac{1}{|\sin(\pi a/p)|}\le \frac{1}{\sin(\pi /p)}< \frac{p}{2},
$$
and
$$
|\xi_{k-1}^{\sigma_{a}}|\ge |\sin(\pi ak/p)|\ge \sin(\pi/p)>\frac{2}{p}.
$$
So we have $|\log|\xi_{k-1}^{\sigma_{a}}||<\log\frac{p}{2}$.
Hence
\begin{equation}
|\log|\eta_{k-1}^{\sigma_{a}}||<\frac{(p-1)\log\frac{p}{2}}{2d}.\notag
\end{equation}

Since we can view $\gal(K/\Q)$ as a quotient group of $\gal(\Q(\zeta_{p})/\Q)$, for any $\sigma\in\gal(K/\Q)$, we have
\begin{equation}\label{xi}
\h(\eta_{k-1}^{\sigma})\le \frac{(p-1)\log 2}{d}\quad \textrm{and} \quad |\log|\eta_{k-1}^{\sigma}||<\frac{(p-1)\log\frac{p}{2}}{2d}.
\end{equation}

\subsection{Height of $\eta_{0}$}

Following the method in Section \ref{height}, we have $ \h(1-\zeta_{p}^{\sigma_{a}})\le \log 2$. So
\begin{equation}
\h(\eta_{0}^{\sigma_{a}})\le \frac{12p(p-1)\log 2}{d}.\notag
\end{equation}

First we have $|1-\zeta_{p}^{\sigma_{a}}|\le 2$. Second we have
$$
|1-\zeta_{p}^{\sigma_{a}}|^{2}\ge 2-2\cos\frac{\pi}{p}
=4\left(\sin\frac{\pi}{2p}\right)^{2}>\left(\frac{2}{p}\right)^{2}.
$$
So we have $|\log|1-\zeta_{p}^{\sigma_{a}}||<\log\frac{p}{2}$. Hence
\begin{equation}
|\log|\eta_{0}^{\sigma_{a}}||<\frac{12p(p-1)\log\frac{p}{2}}{d}.\notag
\end{equation}

Since we can view $\gal(K/\Q)$ as a quotient group of $\gal(\Q(\zeta_{p})/\Q)$, for any $\sigma\in\gal(K/\Q)$, we obtain
\begin{equation}
\h(\eta_{0}^{\sigma})\le \frac{12p(p-1)\log 2}{d}\quad \textrm{and} \quad |\log|\eta_{0}^{\sigma}||<\frac{12p(p-1)\log\frac{p}{2}}{d}.
\end{equation}

\subsection{Height of $|\Upsilon_{c,\sigma}|$}

Recall that $\Upsilon_{c,\sigma}=\gamma_{c,\sigma}/\eta_{0}^{\sigma}$, $\sigma\in\gal(K/\Q)$ and
$$
\gamma_{c,\sigma}=\prod\limits_{\substack{(a_{1},a_{2})\in \OO\sigma\sigma_{c}\\a_{1}=0}}(1-e^{2i\pi a_{2}})^{12p}.
$$
Notice the description of $\OO$ in Section \ref{X_H1}, we have $|\{(a_{1},a_{2})\in\OO\sigma\sigma_c:a_{1}=0\}|\le2|H|= \frac{2(p-1)}{d}$.
Following the method in Section \ref{height}, we get
\begin{equation}
\h(\gamma_{c,\sigma})\le \frac{24p(p-1)\log 2}{d}.
\notag
\end{equation}

Since $\Upsilon_{c,\sigma}=\gamma_{c,\sigma}/\eta_{0}^{\sigma}$, we have
\begin{equation}
\h(\Upsilon_{c,\sigma})\le \h(\gamma_{c,\sigma})+\h(\eta_{0}^{\sigma})\le \frac{36p(p-1)\log 2}{d}.
\notag
\end{equation}
Noticing that $|\Upsilon_{c,\sigma}|^{2}=\Upsilon_{c,\sigma}\bar{\Upsilon}_{c,\sigma}$, we get
\begin{equation}
\h(|\Upsilon_{c,\sigma}|)\le \frac{36p(p-1)\log 2}{d}.
\end{equation}

Since $a_{1}=0$, we have $a_{2}\in \{\frac{1}{p},\cdots,\frac{p-1}{p}\}$. First we have
$|1-e^{2i\pi a_{2}}|\le 2$. Second
$$
|1-e^{2i\pi a_{2}}|^{2}=2(1-\cos2\pi a_{2})\ge 2(1-\cos\pi/p)=4\sin^{2}\frac{\pi}{2p}\ge \frac{4}{p^{2}}.
$$
So we have $|\log|1-e^{2i\pi a_{2}}||\le \log\frac{p}{2}$, and then
$$
|\log|\gamma_{c,\sigma}||\le \frac{24p(p-1)\log\frac{p}{2}}{d}.
$$
Hence we have
\begin{equation}
|\log|\Upsilon_{c,\sigma}||\le \frac{36p(p-1)\log\frac{p}{2}}{d}.
\end{equation}

\subsection{Calculation of $\Omega$} \label{omega}
Recall that $\Omega=A_{1}\cdots A_{d}$, where
$$
A_{k}\ge \max\{\frac{p-1}{2}\h(\alpha_{k}), |\log \alpha_{k}|, 0.16\}, \quad 1\le k\le d.
$$

If $\ord_{c}V=0$, then $\alpha_{k}=|\eta_{k}|=\pm \eta_{k}$, $1\le k\le d-1$, and $\alpha_{d}=|\Upsilon_{c,1}|^{-1}$.
Then for $1\le k\le d-1$, we can choose $A_{k}=p^{2}/d$. For $A_{d}$, we can choose
$A_{d}=36p^{3}/d$.

If $\ord_{c}V\ne 0$, then $\alpha_{k}=\left|\frac{\eta_{k}^{\ord_{c}V^{\sigma}}}{(\eta_{k}^{\sigma})^{\ord_{c}V}}\right|$, $1\le k\le d-1$,
and $\alpha_{d}=\left|\frac{\Upsilon_{c,1}^{\ord_{c}V^{\sigma}}}{\Upsilon_{c,\sigma}^{\ord_{c}V}}\right|^{-1}$.
For $1\le k\le d-1$, combining Lemma \ref{order} we can choose $A_{k}=p^{6}/d^{2}$. For $A_{d}$, we can choose $A_{d}=36p^{7}/d^{2}$.

Therefore, we can choose
\begin{equation}
\Omega=36p^{6d+1}/d^{2d}.
\end{equation}

\subsection{Calculation of $B_{0}$}

For our purpose we need to calculate $\delta,\beta$ and $\kappa$. In fact, all we want to do is to get a bound for $|\alpha_{k\ell}|$,
$1\le k,\ell\le d-1$.

Let $R_{K}$ be the regulator of $K$. By \cite[Lemma 4.15]{Wa}, we have $|\det A|\ge mR_{K}$.
Applying  \cite[Theorem 2]{CF} to the field extension $K/\mathbb{Q}$, we have
$R_{K}>0.32.$ So we get
$$
|\det A|>0.32m.
$$

Notice that $\alpha_{k\ell}=\frac{1}{\det A}A_{\ell k}$, where $A_{lk}$ is the relative cofactor.
The reader should not confuse the matrix $A$, the constants $A_{k}$ introduced in Section \ref{Using Baker} and the cofactors $A_{lk}$.

By Hadamard's inequality and (\ref{xi}), we have
$$
|A_{\ell k}|\le \left[\frac{(p-1)\sqrt{d-2}\log\frac{p}{2}}{2d}\right]^{d-2}.
$$
Then we have
\begin{align*}
|\alpha_{k\ell}|
&<\left[\frac{(p-1)\sqrt{d-2}\log\frac{p}{2}}{2d}\right]^{d-2}\cdot \frac{1}{0.32m}\\
&<\left(p\sqrt{p}\log p\right)^{\frac{p-1}{2}-2}/m\\
&=p^{\frac{3p-15}{4}}(\log p)^{\frac{p-5}{2}}/m.
\end{align*}
Hence, we obtain
\begin{align*}
&\delta<p^{\frac{3p-3}{4}}(\log p)^{\frac{p-5}{2}},\\
&\beta<36p^{\frac{3p-7}{4}}(\log p)^{\frac{p-3}{2}},\\
&\kappa<p^{\frac{3p-11}{4}}(\log p)^{\frac{p-5}{2}}/m.
\end{align*}

Notice that $d\le (p-1)/2$ and $p\ge 7$, we get $C_{1}(d)\le p^{p+8}$.
Therefore, we have
$$
K_{1}<p^{5p+9}(\log p)^{p-1},  \qquad K_{2}<4p^{5p+9}(\log p)^{p-1},
$$
and then
\begin{equation}
B_{0}<16p^{5p+10}(\log p)^{p},\qquad
1+\log B_{0}< 8p\log p.
\notag
\end{equation}

\subsection{Final results}

Finally, by (\ref{j(P)}) we can get an explicit bound for $\log|j(P)|$ as follows
\begin{align*}
\log|j(P)|&<2pC_{1}(d)\Omega(\frac{p-1}{2})^{2}(1+\log \frac{p-1}{2})(1+\log B_{0})\\
&<C(d)p^{6d+5}(\log p)^{2},
\end{align*}
where $C(d)=30^{d+5}\cdot d^{-2d+4.5}$. Hence we obtain Theorem \ref{main1}.

If we choose $d=(p-1)/2$, applying the bound $p-1\ge 6p/7$ and a few numerical computations, we can get
Theorem \ref{main2}.

\section{The case $\Lambda =0$}
\label{specialcase}
In this section, we suppose that $\Lambda=0$. Then we will obtain a better bound for $\log|j(P)|$ than Theorem \ref{main1}.

First we assume that $\ord_{c}V=0$, i.e. $\ord_{c}U=0$. Then we have $|U(P)|=|\gamma_{c}|$. Since $U(P)$ and $\gamma_{c}$ are real, we have
$U(P)^{2}=\gamma_{c}^{2}$, i.e. $U^{2}(P)=\gamma_{c}^{2}$.

Recall $\Omega_{c}$ and the $q$-parameter $q_c$ mentioned in Section \ref{X_H}.
Let $v$ be an absolute value of $\Q(\zeta_{p})$ normalized to extend a standard absolute value on $\Q$.
for the modular function $U^{2}$, we get the following lemma.
\begin{lemma}
There exist an integer function $f(\cdot)$ with respect to $q_{c}$ and $\lambda_{1}^{c}, \lambda_{2}^{c}, \lambda_{3}^{c}\cdots\in \Q(\zeta_{p})$ such that the following identity holds in $\Omega_{c}$,
\begin{equation}\label{Taylor1}
\log\frac{U^{2}(q_c)}{\gamma_{c}^{2}q_{c}^{\frac{2\ord_{c}U}{p}}}=2\pi f(q_{c})i+\sum\limits_{k=1}^{\infty}\lambda_{k}^{c}q_{c}^{k/p},
\end{equation}
and
\begin{equation}
|\lambda_{k}^{c}|_{v}\le\left\{ \begin{array}{ll}
                 |k|_{v}^{-1} & \textrm{if $v$ is finite},\\
                 48p^{2}(k+p) & \textrm{if $v$ is infinite}.
                 \end{array} \right.
\notag
\end{equation}
In particular, for every $k\ge 1$ we have
$$
\h(\lambda_{k}^{c})\le \log(48p^{3}+48kp^{2})+\log k.
$$
\end{lemma}
\begin{proof}
By definition, we have
\begin{equation}\label{U(q_c)}
\frac{U^{2}(q_c)}{\gamma_{c}^{2}q_c^{\frac{2\ord_{c}U}{p}}}=
\prod\limits_{{\rm\bf a}\in \OO\sigma_{c}}\prod\limits_{\substack{n=0\\n+a_{1}\ne 0}}^{\infty}(1-q_{c}^{n+a_{1}}e^{2\pi ia_{2}})^{24p}
\prod\limits_{n=0}^{\infty}(1-q_{c}^{n+1-a_{1}}e^{-2\pi ia_{2}})^{24p}.
\end{equation}
Since
$$
\sum\limits_{{\rm\bf a}\in \OO\sigma_{c}}\left(\sum\limits_{\substack{n=0\\n+a_{1}\ne 0}}^{\infty}24p|q_{c}|^{n+a_{1}}+
\sum\limits_{n=0}^{\infty}24p|q_{c}|^{n+1-a_{1}}\right)
$$
is convergent, it follows from \cite[Chapter 5 Section 2.2 Theorem 6]{Ahlfors} that the right-hand side of (\ref{U(q_c)}) is absolutely convergent.
Then we can write it as the form $\prod\limits_{n=1}^{\infty}(1+d_{n})$ such that $\prod\limits_{n=1}^{\infty}(1+d_{n})$ is absolutely convergent.
Hence, \cite[Chapter 5 Section 2.2 Theorem 5]{Ahlfors} gives
\begin{align*}
&\log \frac{U^{2}(q_c)}{\gamma_{c}^{2}q_c^{\frac{2\ord_{c}U}{p}}}\\
&=2\pi f(q_{c})i+
\sum\limits_{{\rm\bf a}\in \OO\sigma_{c}}\left(\sum\limits_{\substack{n=0\\n+a_{1}\ne 0}}^{\infty}24p\log(1-q_{c}^{n+a_{1}}e^{2\pi ia_{2}})+
\sum\limits_{n=0}^{\infty}24p\log(1-q_{c}^{n+1-a_{1}}e^{-2\pi ia_{2}})\right).
\end{align*}
Applying the Taylor expansion of the logarithm function to the right-hand side of the above formula, we get the desired identity (\ref{Taylor1}).
%formula for $\log\frac{U^{2}(q_c)}{\gamma_{c}^{2}q_c^{\frac{2\ord_{c}U}{p}}}$.

For a fixed non-negative integer $n$ (where we assume $n>0$ if $a_{1}=0$), write
\begin{equation}\label{Taylor}
\log(1-q_{c}^{n+a_{1}}e^{2\pi ia_{2}})=\sum\limits_{k=1}^{\infty}\beta_{k}q^{k/N}.
\notag
\end{equation}
An immediate verification shows that
\begin{equation}
|\beta_{k}|_{v}\le\left\{ \begin{array}{ll}
                 |k|_{v}^{-1} & \textrm{if $v$ is finite},\\
                 1 & \textrm{if $v$ is infinite}.
                 \end{array} \right.
\notag
\end{equation}
Same estimates hold true for the coefficients of the $q$-series for $\log(1-q_{c}^{n+1-a_{1}}e^{-2\pi ia_{2}})$.

For each ${\rm\bf a}\in \OO\sigma_{c}$, the number of coefficients in the $q$-series for $\log(1-q_{c}^{n+a_{1}}e^{2\pi ia_{2}})$ which may contribute to $\lambda_{k}^{c}$ (those with $0\le n\le k/p$) is at most $k/p+1$, and the same is true for the $q$-series for $\log(1-q_{c}^{n+1-a_{1}}e^{-2\pi ia_{2}})$. The bound for $|\lambda_{k}^{c}|_{v}$ now follows by summation.
\end{proof}

\begin{corollary}\label{Taylor2}
With the assumption $\ord_{c}U=0$, we have $\lambda_{k}^{c}\ne 0$ for some $k\le p^{5}$.
\end{corollary}
\begin{proof}
Since $\ord_{c}U=0$ and $U$ is not a constant, there must exist some $\lambda_{k}^{c}\ne 0$. Under the assumption $\ord_{c}U=0$,
we have $U(c)=\gamma_{c}$, and then $f(q_c(c))=0$ by (\ref{Taylor1}).
We extend the additive valuation $\ord_{c}$ from the field $K(X_{H})$ to the field of formal
power series $K((q_{c}^{1/p}))$. Then $\ord_{c}q_{c}^{1/p}=1$ and
$\ord_{c}\left(-2\pi f(q_{c})i+\log (U^{2}/\gamma_{c}^{2})\right)\le\ord_{c}\log (U^{2}/\gamma_{c}^{2})=\ord_{c}(U^{2}/\gamma_{c}^{2}-1)$. The latter quantity is bounded by the degree of
$U^{2}/\gamma_{c}^{2}-1$, which is equal to the degree of $U^{2}$.

The degree of $U^{2}$ is equal to $\sum\limits_{c_{0}}\left|\ord_{c_{0}}\, U\right|$,
here the sum runs through all the cusps of $X_{H}$. Then the result follows from Lemma \ref{order}.
\end{proof}

Now we can get a bound for $\log|j(P)|$.
\begin{proposition}
Under the assumptions $\Lambda=0$ and $\ord_{c}U=0$, we have
$$
\log|j(P)|\le p^{2}\log(48p^{12}+48p^{8})+p\log(96p^{2}(p^{5}+p+1))+\log 2.
$$
\end{proposition}
\begin{proof}
Let $n$ be the smallest $k$ such that $\lambda_{k}^{c}\ne 0$. Then $n\le p^{5}$.
We assume that $|q_c(P)|\le 10^{-p}$, otherwise there is nothing to prove. Since $\ord_{c}U=0$
and $U^{2}(P)=\gamma_{c}^{2}$, it follows from (\ref{Taylor1}) that $2\pi f(q_{c}(P))i+\sum\limits_{k=n}^{\infty}\lambda_{k}^{c}q_{c}(P)^{k/p}=0$.

Suppose that $f(q_{c}(P))=0$. Then $|\lambda_{n}^{c}q_{c}(P)^{n/p}|=|\sum\limits_{k=n+1}^{\infty}\lambda_{k}^{c}q_{c}(P)^{k/p}|$.
 On one side, we have
\begin{align*}
|\sum\limits_{k=n+1}^{\infty}\lambda_{k}^{c}q_{c}(P)^{k/p}|\le \sum\limits_{k=n+1}^{\infty}|\lambda_{k}^{c}||q_{c}(P)|^{k/p}
&\le \sum\limits_{k=n+1}^{\infty}48p^{2}(k+p)|q_{c}(P)|^{k/p}\\
&=96p^{2}(n+p+1)|q_{c}(P)|^{(n+1)/p}.
\end{align*}
On the other side, using Product Formula we get
$$
|\lambda_{n}^{c}|\ge e^{-[\Q(\zeta_{p}):\Q]\h({\lambda_{n}^{c}})}\ge (48np^{3}+48n^{2}p^{2})^{-p+1}.
$$
Then we obtain
$$
\log|q_c(P)^{-1}|\le p^{2}\log(48p^{12}+48p^{8})+p\log(96p^{2}(p^{5}+p+1)).
$$
Finally, the desired result follows from (\ref{j(P)0}).

Suppose that $f(q_{c}(P))\ne 0$. Then $2\pi\le|\sum\limits_{k=n}^{\infty}\lambda_{k}^{c}q_{c}(P)^{k/p}|\le 96p^{2}(n+p)|q_{c}(P)|^{n/p}$.
Then we get $\log|q_c(P)^{-1}|\le p\log(96p^{2}(p^{5}+p))$. So we have
$$
\log|j(P)|\le p\log(96p^{2}(p^{5}+p))+\log 2.
$$
\end{proof}

Now we assume that $\ord_{c}V\ne 0$, i.e. $\ord_{c}U\ne 0$. By Lemma \ref{u2}, we can choose a $U$ such that $\ord_{c}U< 0$.
Then we choose a $\sigma$ such that $\ord_{c}U^{\sigma}> 0$. Put $n_{1}=-\ord_{c}U$ and $n_{2}=\ord_{c}U^{\sigma}$.
Since $U(P)$ and $\gamma_{c}$ are real, we have $U(P)^{2n_{2}}U^{\sigma}(P)^{2n_{1}}=\gamma_{c}^{2n_{2}}\gamma_{c,\sigma}^{2n_{1}}$,
i.e. $U^{2n_{2}}(U^{\sigma})^{2n_{1}}(P)=\gamma_{c}^{2n_{2}}\gamma_{c,\sigma}^{2n_{1}}$.
Lemma \ref{rank} guarantees that $U^{2n_{2}}(U^{\sigma})^{2n_{1}}$ is not a constant.

Applying the same method as the above without difficulties, we can also get a better bound than Theorem \ref{main1}.
We omit the details here.

\section*{Acknowledgement}
We are very grateful to our advisor Yuri Bilu for careful reading and lots of stimulating suggestions and helpful discussions, especially for his key suggestion in Section \ref{specialcase}.
We are also grateful to the referee for careful reading and very useful comments.

\end{document}